\documentclass{article}

\usepackage[all]{xy}
\xyoption{arc}
\xyoption{rotate}
\xyoption{2cell}
\UseAllTwocells

\usepackage{graphicx}
\usepackage{color}
\usepackage{amssymb,amsmath}
\usepackage{mathrsfs}
\usepackage{bbm}
\usepackage[text={131mm,200mm},centering]{geometry}

\usepackage{tocloft}
\usepackage[numbers]{natbib}
\usepackage{natbibspacing}
\renewcommand{\bibfont}{\small}

\definecolor{myurlcolor}{rgb}{0,0,0.5}
\usepackage{hyperref}
\hypersetup{colorlinks,
linkcolor=black,
citecolor=black,
urlcolor=myurlcolor}

\usepackage{latexsym}
\usepackage{amssymb,amsmath}
\usepackage{stmaryrd}

\newcommand{\cat}[1]{\mathscr{#1}}
\newcommand{\ovln}[1]{\overline{#1}}
\newcommand{\twid}[1]{\widetilde{#1}}
\newcommand{\slsh}{/\linebreak[0]}
\newcommand{\dblslsh}{//\linebreak[0]}
\newcommand{\dt}{.\linebreak[0]}
\newcommand{\such}{:}
\newcommand{\without}{\setminus}
\newcommand{\epsln}{\varepsilon}
\newcommand{\integers}{\mathbb{Z}}
\newcommand{\du}{\sqcup}
\newcommand{\reals}{\mathbb{R}}
\newcommand{\rationals}{\mathbb{Q}}
\newcommand{\demph}[1]{\textbf{\textup{#1}}}
\newcommand{\iso}{\cong}
\newcommand{\nat}{\mathbb{N}}	
\newcommand{\sub}{\subseteq}
\newcommand{\mg}[1]{\# #1}
\newcommand{\mgb}[1]{\mg{(#1)}}
\DeclareMathOperator{\adj}{adj}
\newcommand{\card}[1]{\left|#1\right|}
\newcommand{\bigcard}[1]{\bigl|#1\bigr|}
\newcommand{\from}{\colon}
\newcommand{\pwr}[1]{\llbracket #1 \rrbracket}
\newcommand{\Lau}[1]{(\!( #1 )\!)}
\newcommand{\inftyball}[2]{V_{#2}(#1)}
\DeclareMathOperator{\sm}{sum}
\newcommand{\cp}{\mathbin{\Box}}
\newcommand{\strongprod}{\mathbin{\boxtimes}}
\newcommand{\jn}{\vee}
\newcommand{\cen}[1]{\begin{array}{c}#1\end{array}}
\newcommand{\edgeless}[1]{\overline{K_{#1}}}
\newcommand{\cell}[4]{\put(#1,#2){\makebox(0,0)[#3]{#4}}}

\usepackage[amsmath,thmmarks]{ntheorem}

\newtheorem{thm}{Theorem}[section]
\newtheorem{propn}[thm]{Proposition}
\newtheorem{lemma}[thm]{Lemma}
\newtheorem{cor}[thm]{Corollary}

\theorembodyfont{\normalfont}

\newtheorem{defn}[thm]{Definition}
\newtheorem{example}[thm]{Example}
\newtheorem{examples}[thm]{Examples}
\newtheorem{remark}[thm]{Remark}
\newtheorem{remarks}[thm]{Remarks}

\theoremstyle{nonumberplain}
\theoremsymbol{\ensuremath{\square}}
\qedsymbol{\ensuremath{\square}}

\newtheorem{proof}{Proof}

\author{Tom Leinster%
\thanks{School of Mathematics, University of Edinburgh,
Edinburgh EH9 3JZ, 
United Kingdom; 
Tom.Leinster@ed.ac.uk.}}  
\title{The magnitude of a graph}
\date{}

\begin{document}

\sloppy
\maketitle

\begin{abstract}
The magnitude of a graph is one of a family of cardinality-like invariants
extending across mathematics; it is a cousin to Euler characteristic and
geometric measure.  Among its cardinality-like properties are
multiplicativity with respect to cartesian product and an
inclusion-exclusion formula for the magnitude of a union.  Formally, the
magnitude of a graph is both a rational function over $\rationals$ and a
power series over $\integers$.  It shares features with one of the most
important of all graph invariants, the Tutte polynomial; for instance,
magnitude is invariant under Whitney twists when the points of
identification are adjacent.  Nevertheless, the magnitude of a graph is not
determined by its Tutte polynomial, nor even by its cycle matroid, and it
therefore carries information that they do not.
\end{abstract}

\section{Introduction}

\smallskip 
\begin{center}%
\parbox{.85\textwidth}{%
\emph{The analogy\ldots\ the two theories, their conflicts and their
  delicious reciprocal reflections, their furtive caresses, their
  inexplicable quarrels\ldots\ Nothing is more fecund than these slightly
  adulterous relationships.}

\hfill Andr\'e Weil~\cite{Weil}}%
\end{center}
\smallskip

\noindent

In many fields of mathematics, there is a canonical measure of size.  Sets
have cardinality, vector spaces have dimension, and topological spaces have
Euler characteristic (whose status as the topological analogue of
cardinality was made explicit by Schanuel~\cite{SchaNSE}).  Convex subsets
of $\reals^n$ have, in fact, one cardinality-like invariant of each
dimension between $0$ and $n$: the intrinsic volumes~\cite{KlRo}, which
when $n = 2$ are the Euler characteristic, perimeter and area.

Many of these cardinality-like invariants arise from a single general
definition.  This general invariant is called magnitude, and here we
investigate its behaviour in the case of graphs.

The full definition of magnitude is framed in the very wide generality of
enriched categories~\cite{MMS}.  Although we will not need that general
definition here, it is instructive to look briefly at how it specializes to
various branches of mathematics, to give context to what we will do for
graphs.

First, one type of enriched category is an ordinary category, and magnitude
of categories (also called Euler characteristic) is very closely linked to
topological Euler characteristic \cite[Propositions~2.11 and~2.12]{ECC}.
The theory of magnitude of categories also extends the theory of
M\"obius inversion in posets, made famous by Rota for its applications in
enumerative combinatorics~\cite{Rota}.

A second type of enriched category arises commonly in algebra, where one
often encounters categories that are `linear' in the sense that their
hom-sets are vector spaces.  In the representation theory of associative
algebras $A$, an important role is played by the linear category of
indecomposable projective $A$-modules.  What is its magnitude?  Under
suitable hypotheses, it is a recognizable homological invariant of $A$.
Specifically, it is $\chi_A(S, S)$, where $\chi_A$ is the Euler form of $A$
and $S$ is the direct sum of the simple $A$-modules~\cite{MFDA}.

Metric spaces provide a third context for
magnitude~\cite{MMS,AMSES,MeckPDM}.  These too can be seen as enriched
categories, and metric magnitude is a previously undiscovered invariant
that appears to encode many classical quantities.  For example, given a
compact subset $X \sub \reals^n$, write $tX = \{tx \such x \in X\}$ and
$|tX|$ for its magnitude.  Meckes has shown that that the asymptotic growth
of $|tX|$ as $t \to \infty$ is equal to the Minkowski dimension of $X$
\cite[Corollary~7.4]{MeckMDC}.  Moreover, a conjecture of Leinster and
Willerton (\cite{AMSES} and \cite[Conjecture~3.5.10]{MMS}) states that when
$X$ is also convex, $|tX|$ is a polynomial in $t$ whose coefficients are
(up to known scale factors) the intrinsic volumes of $X$: the Euler
characteristic, mean width, surface area, volume, and so on.  The magnitude
of metric spaces is also closely related to certain measures of entropy and
of biological diversity~\cite{METAMB}, and admits a further
potential-theoretic interpretation~\cite{MeckMDC}.

Graphs are metric spaces, with distance between vertices measured as the
length of a shortest path.  Among their special properties is that
distances are integers.  As we shall see, this has the consequence that for
a graph $G$, the magnitude $|tG|$ is a rational function of $q = e^{-t}$
over $\rationals$.  (It can also be expressed as a power series in $q$ over
$\integers$.)  We write it as $\mg{G} = \mg{G}(q)$ to avoid confusion with
the usage of $|G|$ for the number of vertices of $G$, while still evoking
the analogy with cardinality.

Among the cardinality-like properties of magnitude are that 
\begin{equation}        
\label{eq:cp-intro}
\mgb{G \cp H} = \mg{G} \cdot \mg{H}
\end{equation}
where $\cp$ denotes the cartesian product of graphs (defined below), and
that 
\begin{equation}        
\label{eq:union-intro}
\mgb{G \cup H} = \mg{G} + \mg{H} - \mgb{G \cap H}
\end{equation}
under certain hypotheses.  The trivial invariant `number of vertices' also
satisfies these equations, and indeed, the number of vertices of $G$ can be
recovered from its magnitude as $\mg{G}(0)$; but of course, magnitude
is much more informative than that.

The information conveyed by magnitude appears to be quite different from
that conveyed by existing graph invariants.  For instance, the Tutte
polynomial~\cite{Tutt} is perhaps the most important graph invariant of
all, and many other graph invariants are specializations of it, but
magnitude is not; it is not even determined by the graph's cycle matroid.
This is trivial for disconnected graphs, since the graph with $n$ vertices
and no edges has magnitude $n$ but Tutte polynomial $1$ and trivial cycle
matroid.  However, magnitude is not a specialization of the Tutte
polynomial even for connected graphs.  For example, the graphs
\[
\cen{\includegraphics[width=4em]{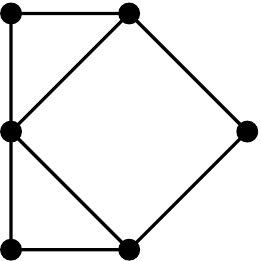}}
\qquad
\cen{\includegraphics[width=4em]{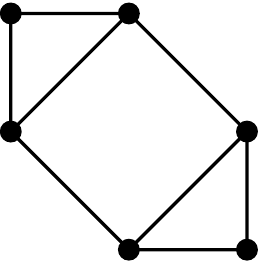}}
\]
have the same cycle matroid, hence also the same Tutte polynomial, the same
number of proper vertex colourings by any given number of colours, the same
number of spanning trees, the same connectivity, the same girth, etc.; but
their magnitudes are different.  Conversely, there are graphs with the same
magnitude that are easily distinguished by well-known graph invariants
(Example~\ref{eg:Ks-same-mag}).  In that sense, magnitude seems to capture
genuinely new aspects of a graph, at the same time as having uniquely good
cardinality-like properties.

We prove two main theorems.  The first is the inclusion-exclusion
formula~\eqref{eq:union-intro} (Theorem~\ref{thm:pjn-union}).  For this we
must impose some hypotheses.  Indeed,
Lemma~\ref{lemma:cant-have-everything} shows that there is \emph{no}
nontrivial graph invariant that is fully cardinality-like in the sense of
satisfying both~\eqref{eq:cp-intro} and~\eqref{eq:union-intro} without
restriction.  But the hypotheses we impose are mild enough to include, for
instance, the case where all the graphs involved are trees, and the case
where $G \cap H$ consists of a single vertex.

It follows that when we join a vertex of a graph $G$ to a vertex of a graph
$H$ to form a new graph $G \jn H$, the magnitude of $G \jn H$ depends only
on the magnitudes of $G$ and $H$, not the vertices chosen.  This is an
invariance property that magnitude shares with the Tutte polynomial.
Another important property of the Tutte polynomial is invariance under
Whitney twists (Figure~\ref{fig:Whitney}, page~\pageref{fig:Whitney}).
This means the following: given graphs $G$ and $H$, each with two chosen,
distinct vertices, we may form new graphs $X$ and $Y$ by gluing $G$ to $H$
at the chosen vertices one way round or the other; then the Tutte
polynomials of $X$ and $Y$ are equal.  Our second main theorem
(Theorem~\ref{thm:Whitney}) is that this is also true for magnitude,
provided that there is an edge between the chosen vertices of either $G$ or
$H$.

Speyer and Willerton showed that even in the case of connected graphs, this
last hypothesis cannot be dropped~\cite{SpeyTPMF,WillTPMF}.  It follows
that magnitude is not a specialization of the Tutte polynomial.

This paper is laid out as follows.  In Section~\ref{sec:defn}, we define
the magnitude of a graph, expressing it as both a rational function and a
power series over $\integers$.  Section~\ref{sec:basic} sets out the most
basic properties and examples of magnitude, including a simple formula for
the magnitude of any graph whose automorphism group acts transitively on
vertices.  We prove that magnitude has some basic cardinality-like
properties.  Viewing $\mg{G}$ as a power series over $\integers$, we also
answer the question: what do the coefficients count?

The remaining two sections prove the two main results: the
inclusion-exclusion theorem (Section~\ref{sec:union}) and the theorem on
invariance under Whitney twists (Section~\ref{sec:Whitney}).  Although both
concern the magnitude of the union of two graphs, the latter is not a
special case of the former, as noted after the statement of
Theorem~\ref{thm:Whitney}.

Recent work of Hepworth and Willerton (in preparation) defines a homology
theory of graphs, of which magnitude is the Euler characteristic.  Their
homology theory is a categorification of graph magnitude in the same sense
that Khovanov's homology theory of knots is a categorification of the Jones
polynomial~\cite{Khov}.  For example, the multiplicativity
property~\eqref{eq:cp-intro} for magnitude of graphs can be derived from a
K\"unneth theorem for magnitude homology of graphs, and similarly, the
inclusion-exclusion formula for magnitude (Theorem~\ref{thm:pjn-union})
lifts to a Mayer--Vietoris theorem in homology.

\section{The definition}
\label{sec:defn}

Here we define the magnitude of a graph, showing that it can be expressed
as either a rational function over $\rationals$ or a power series over
$\integers$.  We also show how to calculate magnitude.

Our conventions are these.  A \demph{graph} is a finite, undirected graph
with no loops or multiple edges.  Graphs may be disconnected or even have
isolated vertices.  Given a graph $G$, we write $V(G)$ for the set of
vertices, $E(G)$ for the set of edges, $v(G)$ for the order of $G$ (the
number of vertices), $e(G)$ for its size (the number of edges), and $k(G)$
for the number of connected-components.  We write $x \in G$ for $x \in
V(G)$.

For vertices $x$ and $y$ of a graph $G$, let $d_G(x, y)$ or $d(x, y)$
denote the length of a shortest path between $x$ and $y$, taken to be
$\infty$ if there is no such path.  This defines a metric on the set of
vertices, provided that we relax the definition of metric space to allow
$\infty$ as a distance.

We now define the magnitude of a graph $G$.  Write $\integers[q]$ for the
polynomial ring over the integers in one variable $q$.  Let $Z_G = Z_G(q)$
be the square matrix over $\integers[q]$ whose rows and columns are indexed
by the vertices of $G$, and whose $(x, y)$-entry is
\[
Z_G(q)(x, y) = q^{d(x, y)}
\]
($x, y \in G$), where by convention $q^\infty = 0$.  Since $Z_G(0)$ is the
identity matrix, the polynomial $\det(Z_G(q))$ has
constant term $1$.  In particular, $\det(Z_G(q))$ is nonzero in the field
$\rationals(q)$ of rational functions over $\rationals$, and so is
invertible there.  It follows that $Z_G(q)$ is invertible as a matrix over
$\rationals(q)$.

\begin{defn}
The \demph{magnitude} of a graph $G$ is
\[
\mg{G}(q) 
=
\sum_{x, y \in G} (Z_G(q))^{-1}(x,y)
\in 
\rationals(q).
\]
We usually abbreviate $\mg{G}(q)$ as $\mg{G}$.
\end{defn}

Writing $\sm(M)$ for the sum of all the entries of a matrix $M$, and
$\adj(M)$ for the adjugate of $M$, we have
\begin{equation}        
\label{eq:mag-frac}
\mg{G}(q)
=
\sm\bigl(Z_G(q)^{-1}\bigr)
=
\frac{\sm\bigl(\adj(Z_G(q))\bigr)}{\det(Z_G(q))}.
\end{equation}
Both the numerator and the denominator are polynomials in $q$ over
$\integers$.

Any rational function over $\rationals$ can be expanded as a Laurent series
over $\rationals$, but $\mg{G}$ has the special property that it is a power
series over $\integers$.  This follows from equation~\eqref{eq:mag-frac},
since the polynomial $\det(Z_G(q))$ has constant term $1$ and is therefore
invertible in the ring $\integers\pwr{q}$ of power series.

(Formally, both $\rationals(q)$ and $\integers\pwr{q}$ are subrings of
$\rationals\Lau{q}$, the ring of Laurent series over $\rationals$.  When we
speak of a rational function being equal to a power series, this means
equality as elements of $\rationals\Lau{q}$.)

\begin{remarks}
\begin{enumerate}
\item
As explained in the introduction, this apparently unmotivated definition is
a special case of the very general definition of the magnitude of an
enriched category \cite[Section~1]{MMS}, which in other contexts produces
a variety of fundamental and classical invariants of size.  

\item
The definition of magnitude also makes sense for directed graphs, with
distance defined non-symmetrically in terms of directed paths.  For
simplicity, we confine ourselves to the undirected case.
\end{enumerate}
\end{remarks}

The magnitude of $G$ is the sum of all the entries of $Z_G(q)^{-1}$, but
it is sometimes useful to consider the individual row-sums.  We define
the \demph{weight} $w_G(x) = w_G(q)(x)$ of a vertex $x$ to be the
corresponding row-sum:
\[
w_G(x) 
= 
\sum_{y \in G} (Z_G(q))^{-1}(x, y) 
\in 
\rationals(q).
\]
The function $w_G\from V(G) \to \rationals(q)$ is called the
\demph{weighting} on $G$, and satisfies the \demph{weighting equations}
\begin{equation}        
\label{eq:wtg}
\sum_{y \in G} q^{d(x, y)} w_G(y) = 1
\qquad
(x \in G).
\end{equation}
(The weighting can alternatively be understood as taking values in
$\integers\pwr{q}$, just as for magnitude itself.)  Magnitude is total
weight: $\mg{G} = \sum_{x \in G} w_G(x)$.  This is loosely analogous to the
Gauss--Bonnet formula for the Euler characteristic of a surface, with
weight playing the role of curvature \cite[Section~2]{ECC}.

We can calculate the magnitude of a graph by finding some function
$\twid{w}_G$ on $V(G)$ satisfying the weighting equations~\eqref{eq:wtg}:

\begin{lemma}   
\label{lemma:wtg-mag}
Let $G$ be a graph and let $\twid{w}_G\from V(G) \to \rationals(q)$ be a
function satisfying the weighting equations.  Then $\twid{w}_G = w_G$ and
$\mg{G} = \sum_{x \in G} \twid{w}_G(x)$.  The same is true when
$\rationals(q)$ is replaced by $\integers\pwr{q}$.
\end{lemma}

\begin{proof}
The matrix $Z_G(q)$ is invertible over $\rationals(q)$, so $w_G\from V(G)
\to \rationals(q)$ is the unique solution to the weighting equations.
Hence $\twid{w}_G = w_G$, giving the result.  The same argument applies
over $\integers\pwr{q}$.
\end{proof}

\section{Basic properties and examples}
\label{sec:basic}

Here we state the most basic facts about magnitude.  We derive formulas for
the magnitudes of vertex-transitive and complete bipartite graphs.  We also
encounter the first pieces of evidence that magnitude of graphs is
analogous to cardinality of sets, proving that magnitude has additivity and
multiplicativity properties similar to those enjoyed by cardinality.

When the magnitude of a graph is expressed as a power series, its
coefficients are integers.  We give a formula for them.  From this it will
follow that the magnitude of a graph determines its order and size.  On the
other hand, it determines neither the chromatic number nor the number of
connected-components, as we show.

We begin with the simplest of examples.

\begin{example} 
\label{eg:edgeless}
Let $G$ be a graph with no edges.  Then $Z_G$ is the identity matrix, so
$\mg{G}$ is the order $v(G)$.  This fits with the conception of magnitude
as generalized cardinality: when a graph has no edges, it is essentially
just a set, and magnitude then reduces to cardinality.
\end{example}

It follows that magnitude is not a specialization of the Tutte polynomial,
since the Tutte polynomial of any edgeless graph is $1$.  Less obvious is
that magnitude is not a specialization of the Tutte polynomial for
\emph{connected} graphs.  We prove this in Section~\ref{sec:Whitney}.

A graph is \demph{vertex-transitive} if its automorphism group acts
transitively on vertices.  The following result is a special case of
\cite[Proposition~2.1.5]{MMS}.

\begin{lemma}[Speyer]       
\label{lemma:trans}
Let $G$ be a vertex-transitive graph.  Then
\[
\mg{G}(q)
=
\frac{v(G)}{\sum_{x \in G} q^{d(g, x)}}
\]
for any $g \in G$.
\end{lemma}

\begin{proof}
By transitivity, the sum $s(q) = \sum_{x \in G} q^{d(g, x)}$ is independent
of $g$.  The result follows by applying Lemma~\ref{lemma:wtg-mag} with
$\twid{w}_G(x) = 1/s$ for all $x \in G$.
\end{proof}

In particular, the diameter of a connected vertex-transitive graph can be
recovered as the degree of its magnitude.  

The denominator of the expression in Lemma~\ref{lemma:trans} closely
resembles the weight enumerator of a linear code, a connection discussed in
\cite[Example~2.3.7]{MMS}.

\begin{examples}        
\label{egs:trans}
\begin{enumerate}
\item   
\label{eg:complete}
By Lemma~\ref{lemma:trans}, the complete graph $K_n$ on $n$ vertices has
magnitude
\[
\mg{K_n}
=
\frac{n}{1 + (n - 1)q}
=
n \sum_{k = 0}^\infty (1 - n)^k q^k.
\]

\item   
\label{eg:cycle}
Similarly, the cycle graph $C_n$ on $n$ vertices has magnitude
\begin{align*}
\mg{C_n}        &
=
\frac{n(q - 1)}%
{q^{\lfloor (n + 1)/2 \rfloor} + q^{\lceil (n + 1)/2 \rceil} - q - 1} 
=
\begin{cases}
\frac{n(q - 1)}{(q^{n/2} - 1)(q + 1)}          &\text{if } n \text{ is
  even,} 
\\[.8ex]
\frac{n(q - 1)}{2q^{(n + 1)/2} - q - 1}        &\text{if } n \text{ is odd.}
\end{cases}
\end{align*}
These equations hold for all $n \geq 1$, interpreting $C_1$ as the graph
with just one vertex and $C_2$ as the graph with just one edge.

\item   
\label{eg:Petersen}
The Petersen graph (shown) is also vertex-transitive, so has magnitude as
follows: 
\[
\begin{array}{c}
\includegraphics[height=15mm]{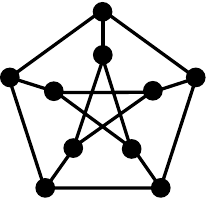}
\end{array}
\qquad
\frac{10}{1 + 3q + 6q^2}
=
10 - 30q + 30q^2 + 90q^3 - 450q^4 + \cdots.
\]
\end{enumerate}
\end{examples}

\begin{example} 
\label{eg:comp-bi}
By direct calculation using Lemma~\ref{lemma:wtg-mag}, the complete
bipartite graph $K_{m, n}$ has magnitude
\[
\mg{K_{m, n}}
=
\frac{(m + n) - (2mn - m - n)q}{(1 + q)(1 - (m - 1)(n - 1)q^2)}.
\]
\end{example}

The cardinality of a disjoint union of sets is the sum of their individual
cardinalities.  The same is true of the magnitude of graphs:

\begin{lemma}   
\label{lemma:coproduct}
Let $G$ and $H$ be graphs.  The magnitude of their disjoint union $G \du H$
is given by $\mgb{G \du H} = \mg{G} + \mg{H}$.
\end{lemma}

\begin{proof}
$Z_{G \du H}$ is the block sum of $Z_G$ and $Z_H$, and the result follows.
\end{proof}

The cardinality of a cartesian product of sets is the product of their
cardinalities, and again, there is an analogous result for the magnitude of
graphs.  Recall that the \demph{cartesian product} $G \cp H$ has $V(G)
\times V(H)$ as its vertex-set, with an edge between $(x, y)$ and $(x',
y')$ if \emph{either} $x = x'$ and $\{y, y'\} \in E(H)$ \emph{or} $y = y'$
and $\{x, x'\} \in E(G)$.

\begin{lemma}  
\label{lemma:cart}
Let $G$ and $H$ be graphs.  The magnitude of their cartesian product is
given by $\mgb{G \cp H} = \mg{G} \cdot \mg{H}$.
\end{lemma}

\begin{proof}
For $x, x' \in G$ and $y, y' \in H$,
\[
d_{G \cp H}((x, y), (x', y')) = d_G(x, x') + d_H(y, y')
\]
and so
\[
Z_{G \cp H}((x, y), (x', y')) = Z_G(x, x') \cdot Z_H(y, y').
\]
So $Z_{G \cp H}$ is the Kronecker product of $Z_G$ and $Z_H$, which implies
that $Z_{G \cp H}^{-1}$ is the Kronecker product of $Z_G^{-1}$ and
$Z_H^{-1}$.  The result follows.
\end{proof}

\begin{example} 
\label{eg:Ks-same-mag}
By Example~\ref{egs:trans}(\ref{eg:complete}) and Lemma~\ref{lemma:cart},
\[
\mgb{K_2 \cp K_3}
=
\mg{K_2} \cdot \mg{K_3}
=
\frac{2}{1 + q} \cdot \frac{3}{1 + 2q} 
=
\frac{6}{1 + 3q + 2q^2}.
\]
So by Example~\ref{eg:comp-bi}, $K_2 \cp K_3$ has the same magnitude as
$K_{3, 3}$, a graph with a different chromatic number.  The chromatic
number cannot, therefore, be derived from the magnitude, even for connected
graphs; hence the Tutte polynomial cannot be either.  We prove the converse
in Section~\ref{sec:Whitney}.
\end{example}

\begin{remarks}
\begin{enumerate}
\item
There is an unfortunate clash of terminology for graph products.  For any
symmetric monoidal category $\cat{V}$, the category of $\cat{V}$-enriched
categories carries a tensor product \cite[Section~1.4]{KellBCE}.  Taking
$\cat{V} = (\nat, \geq, +, 0)$, this gives a tensor product of graphs,
which is what graph theorists call the cartesian product, $\cp$.  On the
other hand, the category of graphs also has what category theorists call a
product, or, for emphasis, a cartesian product; this is what graph
theorists sometimes call the tensor product,
$\times$~\cite[Section~6.3]{GoRo}.

\item
Neither the magnitude of $G \times H$ nor that of the strong product $G
\strongprod H$ \cite[Section~7.15]{GoRo} is determined by the magnitudes of
$G$ and $H$.  Indeed, by Example~\ref{eg:Ks-same-mag}, it is enough to show
that
\[
\mgb{K_2 \times (K_2 \cp K_3)} \neq \mgb{K_2 \times K_{3, 3}},
\qquad
\mgb{K_2 \strongprod (K_2 \cp K_3)} \neq \mgb{K_2 \strongprod K_{3, 3}},
\]
and this is easily done using Lemma~\ref{lemma:trans}.
\end{enumerate}
\end{remarks}

We saw in Section~\ref{sec:defn} that the magnitude of a graph can be
expressed as a power series with integer coefficients.  Those coefficients
can be described explicitly:

\begin{propn}   
\label{propn:series}
For any graph $G$,
\[
\mg{G}(q)
= 
\sum_{k = 0}^\infty (-1)^k 
\sum_{x_0 \neq x_1 \neq \cdots \neq x_k} 
q^{d(x_0, x_1) + \cdots + d(x_{k - 1}, x_k)}
\in 
\integers\pwr{q}
\]
where $x_0, \ldots, x_k$ denote vertices of $G$.  That is, writing
$\mg{G}(q) = \sum_{n = 0}^\infty c_n q^n \in \integers\pwr{q}$,
\[
c_n 
= 
\sum_{k = 0}^n (-1)^k 
\bigcard{\bigl\{(x_0, \ldots, x_k) \such x_0 \neq x_1 \neq \cdots \neq x_k,\ 
d(x_0, x_1) + \cdots + d(x_{k - 1}, x_k) = n\bigr\}}.
\]
\end{propn}

\begin{proof}
The two statements are trivially equivalent; we prove the first.  For $x
\in G$, define $\twid{w}_G(x) \in \integers\pwr{q}$ by
\[
\twid{w}_G(x)
=
\sum_{k = 0}^\infty (-1)^k 
\sum_{x = x_0 \neq x_1 \neq \cdots \neq x_k} 
q^{d(x_0, x_1) + \cdots + d(x_{k - 1}, x_k)}.
\]
We show that $\twid{w}_G$ satisfies the weighting equations.  The result
then follows from Lemma~\ref{lemma:wtg-mag}.

To verify the weighting equations, let $x \in G$.  Then
\begin{align*}
\sum_{y \in G} q^{d(x, y)} \twid{w}_G(y)        &
=
\twid{w}_G(x) 
+ 
\sum_{y\colon y \neq x} q^{d(x, y)} \twid{w}_G(y) \\
&
=
\twid{w}_G(x) 
+ \sum_{k = 0}^\infty (-1)^k \sum_{x \neq y_0 \neq \cdots \neq y_k}
q^{d(x, y_0) + d(y_0, y_1) + \cdots + d(y_{k - 1}, y_k)},
\end{align*}
which cancels to give $1$, as required.  
\end{proof}

Proposition~\ref{propn:series} bears a formal resemblance to Philip Hall's
formula for M\"obius inversion in posets (\cite[Proposition~6]{Rota} and
\cite[Proposition~3.8.5]{StanEC1}), as well as the classical alternating
sum formula for the Euler characteristic of a topological space.  The three
formulas are connected by the notion of the Euler characteristic of a
category \cite[Corollary~1.5 and Proposition~2.11]{ECC}, of which magnitude
is the graph-theoretic analogue.

\begin{cor}     
\label{cor:v-e}
Let $G$ be a graph.  Then $v(G) = \mg{G}(0)$ and $e(G) = -\frac{1}{2}
\frac{d}{dq}\mg{G}(q)\Big|_{q = 0}$.
\end{cor}

\begin{proof}
In the notation of Proposition~\ref{propn:series}, $c_0 = v(G)$ and $c_1 =
-2e(G)$. 
\end{proof}

In particular, magnitude determines both order and size.  Unlike the Tutte
polynomial, it even determines order for disconnected graphs.

\begin{remark}
It follows from Proposition~\ref{propn:series} that $c_0 \geq 0$, $c_1 \leq
0$, and
\[
c_2 
= 
\card{ \{ (x, y, z) \such d(x, y) = d(y, z) = 1 \} }
-
\card{ \{ (x, z) \such d(x, z) = 2 \} }
\geq 
0,
\]
suggesting that the coefficients $c_n$ alternate in sign indefinitely.
However, the Petersen graph (Example~\ref{egs:trans}(\ref{eg:Petersen}))
shows that this is not true in general.
\end{remark}

\begin{example} 
\label{eg:Simon} 
In all the examples so far, $\mg{G}(1)$ is the number $k(G)$ of
connected-components of $G$.  Indeed, this is easily proved in the case
where none of the weights of $G$ has a pole at $1$: for the weighting
equations~\eqref{eq:wtg} then imply that at $q = 1$, the weights in each
component sum to $1$.  But $\mg{G}(1) \neq k(G)$ in general, by the
following example of Willerton \cite[Example~2.2.8]{MMS}.  Let $W$ be the
complete graph $K_6$ with a triangle of edges (but no vertices) removed.
By direct calculation,
\[
\mg{W} = \frac{6}{1 + 4q},
\]
giving $\mg{W}(1) = 6/5 \neq 1 = k(W)$.

In fact, there is \emph{no} way to derive the number of
connected-components from the magnitude.  For, writing $mG$ for the
disjoint union of $m$ copies of a graph $G$, we have
\[
\mg{5W}
=
\frac{30}{1 + 4q}
=
\mg{6K_5}
\]
(by Example~\ref{egs:trans}(\ref{eg:complete}) and
Lemma~\ref{lemma:coproduct}), but $k(5W) = 5 \neq 6 = k(6K_5)$.  
\end{example}

\section{The magnitude of a union}
\label{sec:union}

We now develop the analogy between magnitude of graphs and cardinality of
sets.  We have already seen several aspects of this: the magnitude of a
disjoint union is the sum of the magnitudes (Lemma~\ref{lemma:coproduct}),
the magnitude of a cartesian product is the product of the magnitudes
(Lemma~\ref{lemma:cart}), and the magnitude of a graph with no edges is
simply the cardinality of the vertex-set (Example~\ref{eg:edgeless}).  It
is natural, therefore, to ask whether magnitude obeys the
inclusion-exclusion principle.

In fact, it does not, for reasons that have nothing to do with magnitude.
As we shall see, no nontrivial graph invariant behaves wholly like
cardinality.  However, magnitude does satisfy inclusion-exclusion under
reasonably generous hypotheses on the subgraphs concerned.  This is our
first main result, Theorem~\ref{thm:pjn-union}.

Let us first make precise the claim about cardinality-like invariants.  For
a ring $R$, an $R$-valued \demph{graph invariant} is a function $\Phi$
assigning an element $\Phi(G) \in R$ to each graph $G$, in such a way that
$\Phi(G) = \Phi(H)$ whenever $G \iso H$.  It is \demph{multiplicative} if
$\Phi(K_1) = 1$ and $\Phi(G \cp H) = \Phi(G)\cdot \Phi(H)$ for all $G$ and
$H$.  (Here $K_1$ is the one-vertex graph, the unit for $\cp$.)  It
satisfies \demph{inclusion-exclusion} if $\Phi(\emptyset) = 0$ and
\[
\Phi(X) = \Phi(G) + \Phi(H) - \Phi(G \cap H)
\]
whenever $X$ is a graph with subgraphs $G$ and $H$ such that $G \cup H =
X$.

For example, take any ring $R$, and let $\Phi(G) = v(G)$ be the order of
$G$, interpreted as the element $v(G)\cdot 1 = 1 + \cdots + 1$ of $R$.
Then $\Phi$ is a multiplicative $R$-valued graph invariant satisfying
inclusion-exclusion.  The next lemma tells us that under mild assumptions
on $R$, it is the only one.

\begin{lemma}   
\label{lemma:cant-have-everything}
Let $R$ be a ring containing no nonzero nilpotents.  Then the only
multiplicative $R$-valued graph invariant satisfying inclusion-exclusion is
order.
\end{lemma}

\begin{proof}
Let $\Phi$ be a multiplicative $R$-valued graph invariant satisfying
inclusion-exclusion.  Then $\Phi(G \du H) = \Phi(G) + \Phi(H)$ for all $G$
and $H$.  Writing $\edgeless{n}$ for the edgeless graph on $n$ vertices, we
have $\Phi(\edgeless{0}) = \Phi(\emptyset) = 0$ and $\Phi(\edgeless{1}) =
\Phi(K_1) = 1$, so by induction, $\Phi(\edgeless{n}) = n$ for all $n \geq
0$.

Let $X$ be a graph.  Choose an edge $e$ of $X$, write $X'$ for the subgraph
of $X$ containing all the vertices and all the edges except $e$, and write
$H$ for the subgraph of $X$ consisting of just $e$ and its two endpoints.
Then by inclusion-exclusion,
\begin{align*}
\Phi(X) &
= 
\Phi(X') + \Phi(H) - \Phi(X' \cap H)
\\
&
=
\Phi(X') + \Phi(K_2) - \Phi(\edgeless{2}).
\end{align*}
So, writing $\epsln = \Phi(K_2) - 2$, we have $\Phi(X) = \Phi(X') +
\epsln$.  Applying this argument repeatedly gives $\Phi(X) =
\Phi(\ovln{K_{v(X)}}) + \epsln\cdot e(X)$, that is, $\Phi(X) = v(X) +
\epsln\cdot e(X)$.

It remains to show that $\epsln = 0$, which we do by computing $\Phi(C_4)$
in two ways.  On the one hand, $\Phi(C_4) = 4 + 4\epsln$ by the previous
paragraph.  On the other, $C_4 = K_2 \cp K_2$ and $\Phi$ is multiplicative,
so $\Phi(C_4) = (2 + \epsln)^2$.  Comparing the two expressions gives
$\epsln^2 = 0$.  But $R$ has no nonzero nilpotents, so $\epsln = 0$, as
required.

(For an arbitrary ring $R$, the graph invariants satisfying
multiplicativity and inclusion-exclusion are exactly those of the form $v +
\epsln e$ where $\epsln \in R$ with $\epsln^2 = 0$.)
\end{proof}

We already know that magnitude is a multiplicative graph invariant
(Lemma~\ref{lemma:cart}) and that it is not simply the order.  It cannot,
therefore, satisfy inclusion-exclusion.

Nevertheless, we can seek conditions under which the inclusion-exclusion
principle does hold.  Consider a graph $X$ expressed as the union of
subgraphs $G$ and $H$.  Since magnitude is defined in terms of the metric,
it is natural to ask that distances between vertices of $G$ are the same no
matter whether we measure them in $G$ or in $X$, and similarly for $H$ and
$G \cap H$.  We therefore make the following definition.

\begin{defn}    
\label{defn:convex}
A subgraph $U$ of a graph $X$ is \demph{convex} in $X$ if $d_U(u, u') =
d_X(u, u')$ for all $u, u' \in U$. 
\end{defn}

The terminology comes from a useful analogy between graphs and convex sets.
A subgraph of a graph is convex if its shortest-path metric is the same as
its subspace metric.  Analogously, a compact subset of $\reals^n$ is convex
if its shortest-path metric is the same as the subspace metric.  Of course,
the two uses of `path' are different: in the discrete case, a path of
length $D$ is a distance-preserving map out of $\{0, 1, \ldots, D\}$, while
in the continuous case, it is a distance-preserving map out of $[0, D]$.

When a convex set $X \sub \reals^n$ is covered by closed subsets $G$ and
$H$, it is a fact that if $G \cap H$ is convex then so are $G$ and $H$.
Here is the graph-theoretic analogue.

\begin{lemma}   
\label{lemma:convex-subgraphs}
Let $X$ be a graph, and let $G$ and $H$ be subgraphs with $G \cup H =
X$.  If $G \cap H$ is convex in $X$ then $G$ and $H$ are also convex in $X$.
\end{lemma}

\begin{proof}
We prove it for $G$.  Let $g, g' \in G$.  If $d_X(g, g') = \infty$ then
certainly $d_G(g, g') = d_X(g, g')$.  Otherwise, write $n = d_X(g, g') <
\infty$.  We may choose a shortest path $g = x_0, x_1, \ldots, x_n = g'$
from $g$ to $g'$ in $X$ containing the greatest possible number of vertices
of $G$.  Suppose for a contradiction that $x_j \not\in G$ for some $j$.

By Lemma~\ref{lemma:through-intersection} below, we may choose $i$ and $k$
with $0 \leq i < j < k \leq n$ and $x_i, x_k \in G \cap H$.  Then $x_i,
x_{i + 1}, \ldots, x_k$ is a shortest path from $x_i$ to $x_k$ in $X$, so
$d_X(x_i, x_k) = k - i$.  But $G \cap H$ is convex in $X$, so there is a
path $u_i, u_{i + 1}, \ldots u_k$ from $x_i$ to $x_k$ in $G \cap H$.  Hence
\[
g = x_0, x_1, \ldots, x_{i - 1}, 
x_i = u_i, u_{i + 1}, \ldots, u_{k - 1}, u_k = x_k, 
x_{k + 1}, \ldots, x_{n - 1}, x_n = g'
\]
is a shortest path from $g$ to $g'$ in $X$ containing more vertices of
$G$ than the original path.  This is the required contradiction.
\end{proof}

This proof used the following lemma, which is a combinatorial counterpart
of the fact that when $G$ and $H$ are closed subsets of $\reals^n$, any
path from a point of $G$ to a path of $H$ passes through some point of $G
\cap H$.

\begin{lemma}   
\label{lemma:through-intersection}
Let $X$ be a graph, with subgraphs $G$ and $H$ such that $G \cup H =
X$.  Then every path from a vertex in $G$ to a vertex in $H$ contains at
least one vertex in $G \cap H$.
\end{lemma}

\begin{proof}
Let $x_0, x_1, \ldots, x_n$ be a path with $x_0 \in G$ and $x_n \in H$.
Take the largest $i \in \{0, 1, \ldots, n\}$ such that $x_i \in G$.  We
prove that $x_i \in G \cap H$.  If $i = n$, this is immediate.  If not,
then $x_{i + 1} \not\in G$, so $\{x_i, x_{i + 1}\} \not\in E(G)$.  But
$X = G \cup H$, so $\{x_i, x_{i + 1}\} \in E(H)$, so $x_i \in H$, as
required.
\end{proof}

A wrinkle in the analogy between convex sets and graphs is that in a convex
set, there is only one shortest path between each pair of points, but in a
graph, there may be many.  It is arguably more accurate to say that convex
sets are analogous to trees, since shortest paths in a tree are unique.  We
will see that for trees and subtrees, the inclusion-exclusion formula holds
without restriction (Corollary~\ref{cor:glue-trees}, due to Meckes).  The
following example of Willerton~\cite{WillTPMF} shows that for convex
subgraphs of an arbitrary graph, inclusion-exclusion can fail.

\begin{example}[Willerton]      
\label{eg:two-3-cycles}
Let $X$ be the graph formed by gluing two 3-cycles together along an edge.
Then
\[
\mg{X} 
=
\frac{4 - 2q}{1 + 2q - q^2}
\neq
\frac{4 + 2q}{1 + 3q + 2q^2}
=
2\cdot \mg{C_3} - \mg{C_2},
\]
by direct calculation and Example~\ref{egs:trans}(\ref{eg:cycle})
respectively.  So, magnitude does not satisfy the inclusion-exclusion
principle even when all the subgraphs concerned are convex.  
\end{example}

Convexity will be one hypothesis in our inclusion-exclusion theorem.  We
now formulate the other. 

\begin{defn}
Let $U$ be a convex subgraph of a graph $X$.  Write
\[
\inftyball{X}{U}        
=
\bigcup_{u \in U} \{ x \in X \such d(u, x) < \infty\}   
=
\{ x \in X \such x \text{ is connected to some vertex of } U \}.
\]
We say that $X$ \demph{projects} to $U$ (Figure~\ref{fig:projection}) if for
\begin{figure}
\centering
\setlength{\unitlength}{1em}
\begin{picture}(12,9)(-1.5,.4)
\put(0,0){\includegraphics[height=10\unitlength]{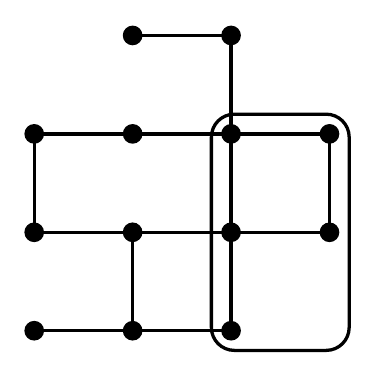}}
\put(0.6,2.9){$x$}
\put(6.3,2.9){$\pi(x)$}
\put(9.7,3.5){$U$}
\put(-1.5,4.7){$X$}
\end{picture}%
\caption{A graph $X$, a subgraph $U$ to which it projects, and 
  the effect of the projection map $\pi$ on a vertex $x$.}
\label{fig:projection}
\end{figure}
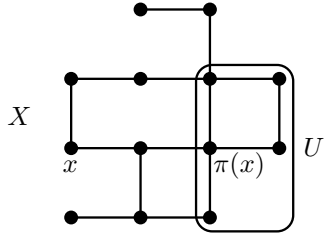
all $x \in \inftyball{X}{U}$, there exists a vertex $\pi(x) \in U$ such
that for all $u \in U$,
\[
d(x, u) = d(x, \pi(x)) + d(\pi(x), u).
\]
\end{defn}

If $X$ projects to $U$ then $\pi(x)$ is uniquely determined by $x$, being
the unique vertex of $U$ closest to $x$.  This defines a \demph{projection}
map $\pi\from \inftyball{X}{U} \to V(U)$.

\begin{example} 
\label{eg:edge-pjn} 
Let $e$ be an edge of a graph $X$.  If the component of $X$ containing $e$
is bipartite, then $X$ projects to the subgraph consisting of $e$ and its
endpoints alone. 
\end{example}

\begin{lemma}   
\label{lemma:pjn-wtgs}
Let $X$ be a graph, and let $U$ be a convex subgraph to which $X$
projects.  Then
\begin{equation}        
\label{eq:fibre-sum}
w_U(u) = \sum_{x \in \pi^{-1}(u)} q^{d(u, x)} w_X(x)
\end{equation}
for each $u \in U$, where $\pi$ denotes the projection.
\end{lemma}

\begin{proof}
Write $\twid{w}_U(u)$ for the right-hand side of~\eqref{eq:fibre-sum}.  We
verify that $\twid{w}_U$ satisfies the weighting equations.  It then
follows from Lemma~\ref{lemma:wtg-mag} that $w_U = \twid{w}_U$.  

Let $u \in U$.  Recalling the convention that $q^\infty = 0$, we have
\begin{align*}
\sum_{v \in U} q^{d(u, v)} \twid{w}_U(v)        &
=
\sum_{v \in U, \ y \in \pi^{-1}(v)}
q^{d(u, v) + d(v, y)} w_X(y)    \\
&
=
\sum_{y \in \inftyball{X}{U}} q^{d(u, \pi(y)) + d(\pi(y), y)} w_X(y)   \\
&
=
\sum_{y \in X} q^{d(u, y)} w_X(y)
=
1,
\end{align*}
as required.
\end{proof}

\begin{thm}     
\label{thm:pjn-union}
Let $X$ be a graph, with subgraphs $G$ and $H$ such that $G \cup H = X$.
Suppose that $G \cap H$ is convex in $X$ and that $H$ projects to
$G \cap H$.  Then 
\[
\mg{X} = \mg{G} + \mg{H} - \mgb{G \cap H}.
\]
\end{thm}

\begin{proof}
We will prove that $w_X = w_G + w_H - w_{G \cap H}$, where on the
right-hand side, the function $w_G$ on $V(G)$ is extended by zero to all of
$V(X)$, and similarly $w_H$ and $w_{G \cap H}$.  The theorem then follows
immediately.

We may unambiguously write $d$ for distance, by
Lemma~\ref{lemma:convex-subgraphs}.  Also, we write $\pi\from
\inftyball{H}{G \cap H} \to V(G \cap H)$ for the projection associated
with $G \cap H \sub H$.

First I claim that for all $g \in G$ and $h \in \inftyball{H}{G \cap
  H}$,
\begin{equation}        
\label{eq:pjn-path}
d(g, h) = d(g, \pi(h)) + d(\pi(h), h).
\end{equation}
If $d(g, h) = \infty$, this is immediate from the triangle inequality.
Otherwise, by Lemma~\ref{lemma:through-intersection}, $d(g, h) = d(g, u) +
d(u, h)$ for some $u \in G \cap H$.  But also
\begin{align*}
d(g, u) + d(u, h)       &
=
d(g, u) + d(u, \pi(h)) + d(\pi(h), h)   \\
&
\geq
d(g, \pi(h)) + d(\pi(h), h)
\geq
d(g, h),
\end{align*}
so equality holds throughout, proving the claim.

We now verify that $w_G + w_H - w_{G \cap H}$ satisfies the weighting
equations for $X$.  These state that for all $x \in X$,
\begin{equation}        
\label{eq:union-main}
\sum_{g \in G} q^{d(x, g)} w_G(g) + \sum_{h \in H} q^{d(x, h)} w_H(h)
- \sum_{u \in G \cap H} q^{d(x, u)} w_{G \cap H}(u)
=
1.
\end{equation}
If $x \in G$ then by Lemma~\ref{lemma:pjn-wtgs}, the left-hand side
of~\eqref{eq:union-main} is
\[
1 
+ \sum_{h \in H} q^{d(x, h)} w_H(h) 
- \sum_{u \in G \cap H, \ h \in \pi^{-1}(u)} 
q^{d(x, u) + d(u, h)} w_H(h),
\]
which by Lemma~\ref{lemma:through-intersection} is equal to
\[
1 
+ \sum_{h \in \inftyball{H}{G \cap H}} q^{d(x, h)} w_H(h) 
- \sum_{h \in \inftyball{H}{G \cap H}} 
q^{d(x, \pi(h)) + d(\pi(h), h)} w_H(h), 
\]
and equation~\eqref{eq:pjn-path} implies that this is equal to $1$.  If $x
\in \inftyball{H}{G \cap H}$ then by equation~\eqref{eq:pjn-path}, the
left-hand side of~\eqref{eq:union-main} is
\begin{align*}
&
q^{d(x, \pi(x))} \sum_{g \in G} q^{d(\pi(x), g)} w_G(g)
+ 1 
- q^{d(x, \pi(x))} \sum_{u \in G \cap H} q^{d(\pi(x), u)} w_{G \cap H}(u)\\
= {}
&
q^{d(x, \pi(x))} + 1 - q^{d(x, \pi(x))} 
= 
1.
\end{align*}
Finally, if $x \in V(H) \without \inftyball{H}{G \cap H}$ then by
Lemma~\ref{lemma:through-intersection}, the left-hand side
of~\eqref{eq:union-main} is $0 + 1 - 0 = 1$.  So 
equation~\eqref{eq:union-main} holds in all cases, giving
$w_X = w_G + w_H - w_{G \cap H}$ by Lemma~\ref{lemma:wtg-mag}, as
required. 
\end{proof}

We record three corollaries.  First, given graphs $G$ and $H$, we may form
their one-point join $G \jn H$, obtained from the disjoint union of $G$ and
$H$ by identifying one vertex of $G$ with one vertex of $H$.  In principle,
the magnitude of $G \jn H$ could depend on the vertices chosen; but, like
the Tutte polynomial, it does not.

\begin{cor}     
\label{cor:join}
Let $G$ and $H$ be graphs.  Then $\mgb{G \jn H} = \mg{G} + \mg{H} - 1$.
\qed
\end{cor}

The Tutte polynomial does not distinguish between the one-point join of two
graphs and their disjoint union: $T_{G \jn H} = T_{G \du H}$.  Magnitude
does: by Corollary~\ref{cor:join} and Lemma~\ref{lemma:coproduct}, $\mgb{G
  \jn H} = \mgb{G \du H} - 1$.

\begin{example}
Consider the following three graphs:
\[
\includegraphics[height=2em]{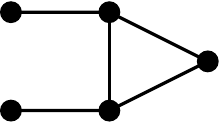}
\qquad
\includegraphics[height=2em]{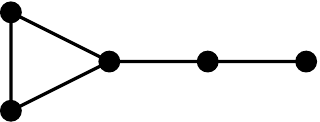}
\qquad
\includegraphics[height=2em]{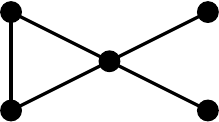}
\]
Using the one-point join operation twice, we can build each of them from
the same pieces, one copy of $K_3$ and two of $K_2$.  So all three have the
same magnitude (as well as the same Tutte polynomial), namely
\[
\mg{K_3} + 2\cdot\mg{K_2} - 2
=
\frac{5 + 5q - 4q^2}{(1 + q)(1 + 2q)}.
\]
\end{example}

\begin{example} 
\label{eg:forests}
Any forest $G$ can be obtained by successively joining edges to the
edgeless graph with one vertex for each component of $G$.  Repeated
application of Corollary~\ref{cor:join} gives
\begin{align*}
\mg{G}  &
= 
k(G) + e(G)\frac{1 - q}{1 + q}  
= 
v(G) - 2e(G)\frac{q}{1 + q}     \\
        &
=
v(G) - 2e(G)q + 2e(G)q^2 - 2e(G)q^3 + \cdots.
\end{align*}
In particular, the magnitude of a tree depends only on the number of edges.
\end{example}

Our second corollary, due to Meckes~\cite{MeckTPMF}, follows from
Example~\ref{eg:forests} or can be proved directly from
Theorem~\ref{thm:pjn-union}.

\begin{cor}[Meckes]       
\label{cor:glue-trees}
Let $X$ be a tree, with subtrees $G$ and $H$ such that $G \cup H = X$.
Then $\mg{X} = \mg{G} + \mg{H} - \mgb{G \cap H}$. \qed
\end{cor}

\begin{cor}     
\label{cor:bip-glue}
Let $G$ be a graph and $H$ a bipartite graph.  Let $X$ be a graph
obtained by identifying some edge of $G$ with some edge of $H$.  Then
\[
\mg{X} 
= 
\mg{G} + \mg{H} - \frac{2}{1 + q}.
\]
\end{cor}

\begin{proof}
This follows from Theorem~\ref{thm:pjn-union} and
Example~\ref{eg:edge-pjn}, using the formula for $\mg{K_2}$ in
Example~\ref{egs:trans}(\ref{eg:complete}).
\end{proof}

\begin{example} 
\label{eg:even-cycle-anywhere}
Corollary~\ref{cor:bip-glue} implies that when an arbitrary graph $G$ has
an even cycle glued onto it by an edge, the magnitude of the resulting
graph does not depend on which edge of $G$ the cycle was glued onto.  This
is false in general for odd cycles, as the next example shows.  
\end{example}

\begin{example} 
\label{eg:houses}
Let $B$ be the graph formed by gluing a 3-cycle to a 4-cycle along an
edge.  By Corollary~\ref{cor:bip-glue}, $\mg{B} = \mg{C_3} + \mg{C_4} -
\mg{C_2}$.  

Now consider gluing a 3-cycle to $B$ along another edge of the 4-cycle.
Depending on which edge of $B$ we glue along, this could produce either of
the two graphs
\[
X = \cen{\includegraphics[width=4em]{Xhouse_mono.pdf}},
\qquad
Y = \cen{\includegraphics[width=4em]{Yhouse_mono.pdf}}.
\]
Neither $B$ nor $C_3$ is bipartite, so Corollary~\ref{cor:bip-glue} does
not apply to either $X$ or $Y$.  However, Theorem~\ref{thm:pjn-union} does
apply to $X$, taking $G = C_3$ and $H = B$.  Thus,
\[
\mg{X} 
=
\mg{C_3} + \mg{B} - \mg{C_2}
=
2\cdot\mg{C_3} + \mg{C_4} - 2\cdot\mg{C_2}
=
\frac{6 + 8q - 2q^2}{1 + 4q + 5q^2 + 2q^3}.
\]
On the other hand, the hypotheses of Theorem~\ref{thm:pjn-union} do not
hold for $Y = C_3 \cup B$.  Nor does the conclusion, since a direct
calculation shows that
\[
\mg{Y}
=
\frac{6 - 4q}{1 + 2q - q^3}
\neq
\mg{X}.
\]
\end{example}

\section{Whitney twists}
\label{sec:Whitney}

Much information about a graph is contained in its cycle matroid.
(See~\cite{Oxle}, for instance.)  Essentially by definition, two graphs $G$
and $H$ have isomorphic cycle matroids if and only if there is a bijection
between their edge-sets with the property that a sequence of edges in $G$
is a cycle exactly when the corresponding sequence in $H$ is a cycle.  In
1933~\cite{Whit2IG}, Whitney showed that two graphs have isomorphic cycle
matroids if and only if one can be transformed into the other by a finite
sequence of moves of the following three types.

The first is \demph{vertex identification}: whenever a graph $X$ can be
decomposed as a disjoint union $G \du H$, and $g$ and $h$ are vertices of
$G$ and $H$ respectively, change $X$ to the graph $G \jn H$ formed by
identifying $g$ with $h$.  The second is the reverse of the first.

\begin{figure}
\centering
\setlength{\unitlength}{1mm}
\begin{picture}(120,35)
\cell{28}{5}{b}{\includegraphics[height=30\unitlength]{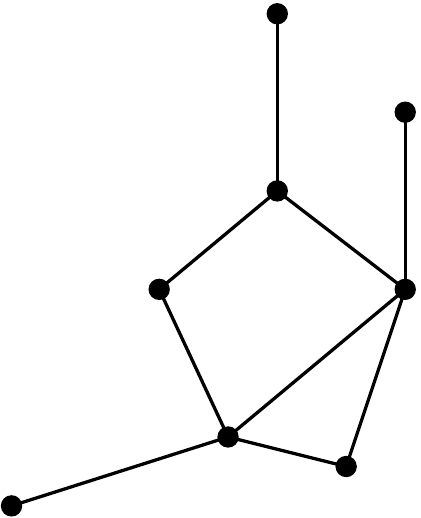}}
\cell{28}{-1}{b}{$X$}
\cell{24.5}{20.5}{c}{$g_+$}
\cell{24.5}{16}{c}{$h_+$}
\cell{42.5}{20.5}{c}{$g_-$}
\cell{41.5}{15}{c}{$h_-$}
\cell{10}{11}{c}{$\left\{\makebox[0em]{\rule{0em}{7\unitlength}}\right.$}
\cell{10}{27}{c}{$\left\{\makebox[0em]{\rule{0em}{10\unitlength}}\right.$}
\cell{6.5}{27}{c}{$G$}
\cell{6.5}{11}{c}{$H$}
\cell{90}{5}{b}{\includegraphics[height=30\unitlength]{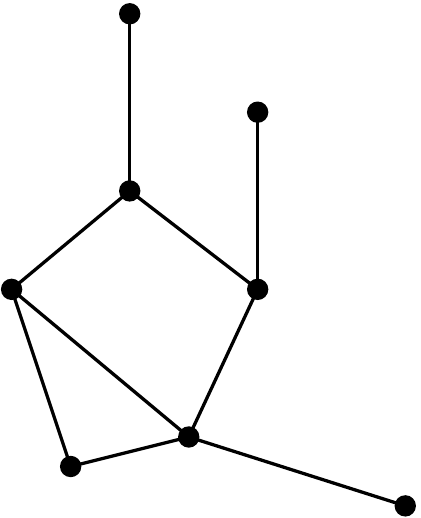}}
\cell{90}{-1}{b}{$Y$}
\cell{78}{20.5}{c}{$g_+$}
\cell{77.5}{15}{c}{$h_-$}
\cell{96}{20.5}{c}{$g_-$}
\cell{95.5}{16}{c}{$h_+$}
\cell{108}{11}{c}{$\left.\makebox[0em]{\rule{0em}{7\unitlength}}\right\}$}
\cell{108}{27}{c}{$\left.\makebox[0em]{\rule{0em}{10\unitlength}}\right\}$}
\cell{111}{27}{c}{$G$}
\cell{111}{11}{c}{$H$}
\end{picture}
\caption{Two graphs $X$ and $Y$ differing by a Whitney twist.}
\label{fig:Whitney}
\end{figure}

The third is the Whitney twist, defined as follows
(Figure~\ref{fig:Whitney}).  Take a graph $G$ equipped with two distinct
distinguished vertices, $g_+$ and $g_-$, and take $H$, $h_+$ and $h_-$
similarly.  Form a new graph $X$ by taking the disjoint union of $G$ and
$H$ then identifying $g_+$ with $h_+$ and $g_-$ with $h_-$ (and, if this
creates a double edge between the points of identification, identifying
those edges).  Define $Y$ similarly, but identifying $g_+$ with $h_-$ and
$g_-$ with $h_+$.  The graphs $X$ and $Y$ are said to differ by a
\demph{Whitney twist}.

By the theorem of Whitney, a graph invariant assigns the same value to
graphs with isomorphic cycle matroids if and only if it is invariant under
vertex identification and Whitney twists.  Now, magnitude is not invariant
under vertex identification, as by Lemma~\ref{lemma:coproduct} and
Corollary~\ref{cor:join},
\[
\mg{(G \du H)} = \mg{G} + \mg{H},
\qquad
\mg{(G \jn H)} = \mg{G} + \mg{H} - 1.
\]
However, these equations imply that $\Phi(G) = \mg{G} - k(G)$ is invariant
under vertex identification, where $k$ is the number of
connected-components.  Moreover, $k$ is invariant under Whitney twists.
Hence $\Phi$ depends only on the cycle matroid if and only if magnitude is
invariant under Whitney twists.

We show here that, in fact, magnitude is not invariant under Whitney
twists; so $\Phi$ does not depend only on the cycle matroid.  Moreover,
since the Tutte polynomial can be defined in terms of the cycle matroid and
is therefore invariant under Whitney twists, magnitude is not a
specialization of the Tutte polynomial.  This is trivially true for
disconnected graphs (by Example~\ref{eg:edgeless}), but the graphs in our
counterexample are connected.

On the other hand, the main result of this section is that magnitude
\emph{is} invariant under Whitney twists when the two points of
identification are adjacent (Theorem~\ref{thm:Whitney}).  In this sense,
$\mg{G} - k(G)$ comes close to depending only on the cycle matroid of $G$.

We begin by exhibiting two graphs that differ by a Whitney twist but do not
have the same magnitude.  This strategy for showing that magnitude is not a
specialization of the Tutte polynomial was suggested by
Speyer~\cite{SpeyTPMF}, and the first example of such a pair was found by
Willerton~\cite{WillTPMF}.  The following proof uses a smaller example.

\begin{propn}[Speyer and Willerton]
\label{propn:not-specialization} 
There exists a pair of connected graphs with
isomorphic cycle matroids (hence the same Tutte polynomial) but different
magnitudes.
\end{propn}

\begin{proof}
The graphs $X$ and $Y$ of Example~\ref{eg:houses} differ by a Whitney
twist, but have different magnitudes.
\end{proof}

Before we prove our main result on Whitney twists, let us fix some
notation.  We work with graphs $X$ and $Y$ obtained from $(G, g_+, g_-)$
and $(H, h_+, h_-)$, as in the definition of Whitney twist stated above.
The vertex of $X$ formed by identifying $g_+$ with $h_+$ will be denoted by
either $g_+$ or $h_+$; thus, $g_+ = h_+$ as vertices of $X$.  We refer to
$g_+ = h_+$ and $g_- = h_-$ as the \demph{gluing points} of $X$, and
similarly for $Y$.

The vertices of $X$ that are not gluing points are in canonical bijection
with the vertices of $Y$ that are not gluing points.  The two gluing points
are adjacent in $X$ if and only if either $g_+$ is adjacent to $g_-$ in $G$
or $h_+$ is adjacent to $h_-$ in $H$.  This in turn is equivalent to the
gluing points being adjacent in $Y$.

\begin{thm}     
\label{thm:Whitney}
Let $X$ and $Y$ be graphs differing by a Whitney twist, and suppose that
the two gluing points are adjacent in $X$ (or equivalently $Y$).  Then
$\mg{X} = \mg{Y}$.
\end{thm}

This was conjectured by Willerton~\cite{WillTPMF}.  In the proof, we do not
attempt to derive any expression for $\mg{X}$ or $\mg{Y}$ in terms of
$\mg{G}$ and $\mg{H}$.  (Example~\ref{eg:two-3-cycles} shows that it is not
given by the inclusion-exclusion formula.)  Instead, we find a
direct relationship between the weightings on $X$ and $Y$.

\begin{proof}
We use the same notation as above, and assume without loss of generality
that $\{g_+, g_-\} \in E(G)$ and $\{h_+, h_-\} \in E(H)$.

Both $G$ and $H$ are convex in both $X$ and $Y$, so we may unambiguously
use the unsubscripted notation $d(a, b)$ when $a$ and $b$ both belong to
$G$ or both belong to $H$.  To describe the other distances in $X$ and $Y$,
it is convenient to introduce some further notation.  For $g \in G$, write
\[
\delta(g) = \min \{ d(g, g_-), d(g, g_+) \},
\]
and similarly $\delta(h)$ for $h \in H$.  Partition $V(G)$ as $G_+ \cup G_0
\cup G_-$, where
\begin{align*}
G_+     &=      \{ g \in G \such d(g, g_+) < d(g, g_-)\},       \\
G_0     &=      \{ g \in G \such d(g, g_+) = d(g, g_-)\},       \\
G_-     &=      \{ g \in G \such d(g, g_+) > d(g, g_-)\},
\end{align*}
and similarly for $H$.  Then for $g \in G$ and $h \in H$, we have
\begin{align*}
d_X(g, h)       &
=
\begin{cases}
\delta(g) + \delta(h) + 1       &
\text{if } (g \in G_+ \text{ and } h \in H_-) 
\text{ or } (g \in G_- \text{ and } h \in H_+)  \\
\delta(g) + \delta(h)   &
\text{otherwise,}
\end{cases}     \\
d_Y(g, h)      &
=
\begin{cases}
\delta(g) + \delta(h) + 1       &
\text{if } (g \in G_+ \text{ and } h \in H_+) 
\text{ or } (g \in G_- \text{ and } h \in H_-)  \\
\delta(g) + \delta(h)   &
\text{otherwise.}
\end{cases}
\end{align*}

We now describe the weighting on $Y$.  Put
\[
u^G_+ = \sum_{g \in G_+} q^{\delta(g)} w_X(g),
\quad
u^G_0 = \sum_{g \in G_0} q^{\delta(g)} w_X(g),
\quad
u^G_- = \sum_{g \in G_-} q^{\delta(g)} w_X(g),
\]
and similarly $u^H_+$, $u^H_0$ and $u^H_-$.  Define $\twid{w}_Y\from V(Y)
\to \rationals(q)$ by $\twid{w}_Y(y) = w_X(y)$ whenever $y$ is not a gluing
point, and
\begin{align}
\twid{w}_Y(g_+)         &
= 
w_X(g_+) - u^G_+ + u^G_-,
\label{eq:wY-plus}      \\
\twid{w}_Y(g_-)         &
= 
w_X(g_-) - u^G_- + u^G_+.
\label{eq:wY-minus}
\end{align}
We will show that $\twid{w}$ satisfies the weighting equations for $Y$,
which by Lemma~\ref{lemma:wtg-mag} implies that $\twid{w}_Y =
w_Y$, hence $\mg{Y} = \mg{X}$.

First I claim that the defining equations~\eqref{eq:wY-plus}
and~\eqref{eq:wY-minus} for $\twid{w}_Y$ are unchanged if we replace $G$ by
$H$ and $g$ by $h$ throughout.  Because of the identifications between
$g_\pm$ and $h_\pm$ in $X$ and in $Y$, this reduces to the claim that
\begin{equation}
\label{eq:sym}
w_X(g_+) - u^G_+ + u^G_- 
=
w_X(h_-) - u^H_- + u^H_+.
\end{equation}

To prove this, note that
\[
V(X) = (G_+ \cup G_0 \cup G_-) \cup (H_+ \cup H_0 \cup H_-),
\]
this union being disjoint except that $G_+ \cap H_+ = \{ g_+ \}$ and $G_-
\cap H_- = \{ g_- \}$.  The weighting equation $\sum_{x \in X} q^{d_X(g_+,
  x)} w_X(x) = 1$ therefore gives
\begin{equation}
\label{eq:pvt-plus}
(u^G_+ + u^G_0 + qu^G_-) 
+ (u^H_+ + u^H_0 + qu^H_-) 
- (w_X(g_+) + qw_X(g_-))
=
1.
\end{equation}
The same is true when $+$ and $-$ are interchanged:
\begin{equation}
\label{eq:pvt-minus}
(qu^G_+ + u^G_0 + u^G_-) 
+ (qu^H_+ + u^H_0 + u^H_-) 
- (qw_X(g_+) + w_X(g_-))
=
1.
\end{equation}
Subtracting~\eqref{eq:pvt-minus} from~\eqref{eq:pvt-plus}
gives~\eqref{eq:sym}, proving the claim.

We now show that $\twid{w}_Y$ satisfies the weighting equations.  By the
symmetry just established, it is enough to show that $\sum_{y \in Y}
q^{d_Y(g, y)} \twid{w}_Y(y) = 1$ whenever $g \in G$.  Let $g \in G$.  We
have
\begin{equation}        
\label{eq:to-vanish}
\sum_{y \in Y} q^{d_Y(g, y)} \twid{w}_Y(y) - 1  
=
\sum_{y \in Y} q^{d_Y(g, y)} \twid{w}_Y(y) 
- \sum_{x \in X} q^{d_X(g, x)} w_X(x),
\end{equation}
and we want to prove that the left-hand side of~\eqref{eq:to-vanish} is
zero.  When $x = y \in G\without\{g_+, g_-\}$, we have $d_Y(g, y) = d_X(g,
x)$ and $\twid{w}_Y(y) = w_X(x)$, so the $x$- and $y$-summands on the
right-hand side cancel out.  The same is true when $x = y \in H_0$.  The
right-hand side is therefore unchanged if each sum is restricted to run
over only $H_+ \cup H_-$.  So by definition of $\twid{w}_Y$, 
the right-hand side is equal to
\begin{equation}
\label{eq:to-vanish-long}
\sum_{h \in (H_+ \cup H_-) \without \{ h_+, h_- \}}
\bigl(q^{d_Y(g, h)} - q^{d_X(g, h)}\bigr) w_X(h)
+ \bigl(q^{d(g, g_+)} - q^{d(g, g_-)}\bigr)\bigl(u^G_- - u^G_+\bigr).
\end{equation}
We must show that this is zero.  If $g \in G_0$ then every summand
in~\eqref{eq:to-vanish-long} vanishes.  If $g \in G_+$
then~\eqref{eq:to-vanish-long} is equal to
\begin{align*}
&
\sum_{h \in H_+ \without \{ h_+ \}} 
\bigl(q^{\delta(g) + \delta(h) + 1} - q^{\delta(g) + \delta(h)}\bigr) 
w_X(h) 
\\
&
{} +
\sum_{h \in H_- \without \{ h_- \}} 
\bigl(q^{\delta(g) + \delta(h)} - q^{\delta(g) + \delta(h) + 1}\bigr) 
w_X(h)
+ 
\bigl(q^{\delta(g)} - q^{\delta(g) + 1}\bigr)\bigl(u^G_- - u^G_+\bigr)  
\\
=\ &
q^{\delta(g)}(q - 1) 
\Bigl\{
\bigl(u^H_+ - w_X(h_+)\bigr) - \bigl(u^H_- - w_X(h_-)\bigr) 
- \bigl(u^G_- - u^G_+\bigr)
\Bigr\} \\
=\ &
0,
\end{align*}
using~\eqref{eq:sym} in the last step.  By symmetry, if $g \in G_-$
then~\eqref{eq:to-vanish-long} is also zero.
Hence~\eqref{eq:to-vanish-long} is zero in all cases, completing the proof.
\end{proof}

\begin{example}
Randomly generate graphs $G$ and $H$, making each pair of vertices adjacent
with probability $p$.  Choose at random a pair of distinct vertices in each
of $G$ and $H$, and glue $G$ and $H$ together at these vertices to form
graphs $X$ and $Y$ differing by a Whitney twist.  The probability that the
gluing points are adjacent in $X$ is $p(2 - p)$, so by
Theorem~\ref{thm:Whitney}, the probability that $\mg{X} = \mg{Y}$ is at
least $p(2 - p)$.  For example, when $p = 1/2$, graphs differing by a
Whitney twist have equal magnitude with probability at least $3/4$.
\end{example}

It may happen that graphs differing by a Whitney twist have the same
magnitude even if the gluing points are not adjacent.  This can occur for
trivial reasons of symmetry, or for other reasons.  For example, the graphs
\[
\includegraphics[width=6em]{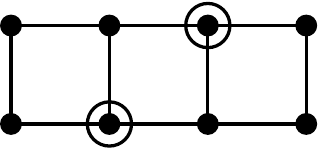}
\qquad\qquad
\includegraphics[width=4.4em]{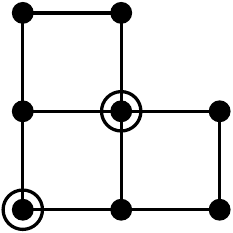}
\]
differ by a Whitney twist, but the gluing points (circled) are not
adjacent, so the hypotheses of Theorem~\ref{thm:Whitney} are not satisfied.
Nevertheless, Example~\ref{eg:even-cycle-anywhere} guarantees that they
have the same magnitude.

\paragraph*{Acknowledgements}  
Thanks to Owen Biesel, Joachim Kock and Mark Meckes for helpful
suggestions.  I am especially grateful to David Speyer and Simon Willerton
for sharing their ideas and allowing me to include
Proposition~\ref{propn:not-specialization}.


\begin{thebibliography}{10}

\bibitem{MFDA}
J.~Chuang, A.~King, and T.~Leinster.
\newblock On the magnitude of a finite dimensional algebra.
\newblock In preparation, 2014.

\bibitem{GoRo}
C.~Godsil and G.~Royle.
\newblock {\em Algebraic Graph Theory}.
\newblock Springer, New York, 2001.

\bibitem{KellBCE}
G.~M. Kelly.
\newblock {\em Basic Concepts of Enriched Category Theory}, volume~64 of {\em
  London Mathematical Society Lecture Note Series}.
\newblock Cambridge University Press, Cambridge, 1982.
\newblock Also \emph{Reprints in Theory and Applications of Categories} 10
  (2005), 1--136.

\bibitem{Khov}
M.~Khovanov.
\newblock A categorification of the {J}ones polynomial.
\newblock {\em Duke Mathematical Journal}, 101:359--426, 2000.

\bibitem{KlRo}
D.~A. Klain and G.-C. Rota.
\newblock {\em Introduction to Geometric Probability}.
\newblock Lezioni Lincee. Cambridge University Press, Cambridge, 1997.

\bibitem{ECC}
T.~Leinster.
\newblock The {E}uler characteristic of a category.
\newblock {\em Documenta Mathematica}, 13:21--49, 2008.

\bibitem{METAMB}
T.~Leinster.
\newblock A maximum entropy theorem with applications to the measurement of
  biodiversity.
\newblock \href{http://arxiv.org/abs/0910.0906}{arXiv:0910.0906}, 2009.

\bibitem{MMS}
T.~Leinster.
\newblock The magnitude of metric spaces.
\newblock {\em Documenta Mathematica}, 18:857--–905, 2013.

\bibitem{NCTPMF}
T.~Leinster.
\newblock {T}utte polynomials and magnitude functions.
\newblock Post at \emph{The $n$-Category Caf{\'e}},
  \href{http://golem.ph.utexas.edu/category/2013/04/tutte_polynomials_and_magn%
itud.html}{http:\dblslsh golem\dt ph\dt utexas\dt edu\slsh category\slsh
  2013\slsh 04\slsh tutte\_polynomials\_and\_magnitud\dt html}, 2013.

\bibitem{AMSES}
T.~Leinster and S.~Willerton.
\newblock On the asymptotic magnitude of subsets of {E}uclidean space.
\newblock {\em Geometriae Dedicata}, 164:287--310, 2013.

\bibitem{MeckPDM}
M.~W. Meckes.
\newblock Positive definite metric spaces.
\newblock {\em Positivity}, 17:733--757, 2013.

\bibitem{MeckTPMF}
M.~W. Meckes.
\newblock Re: {T}utte polynomials and magnitude functions.
\newblock Comment at \cite{NCTPMF}, 2013.

\bibitem{MeckMDC}
M.~W. Meckes.
\newblock Magnitude, diversity, capacities, and dimensions of metric spaces.
\newblock \emph{Potential Analysis}, to appear, 2014.

\bibitem{Oxle}
J.~G. Oxley.
\newblock {\em Matroid Theory}.
\newblock Oxford University Press, Oxford, 1992.

\bibitem{Rota}
G.-C. Rota.
\newblock On the foundations of combinatorial theory {I}: theory of
  {M}{\"o}bius functions.
\newblock {\em Zeitschrift f{\"u}r Wahrscheinlichkeitstheorie und Verwandte
  Gebiete}, 2:340--368, 1964.

\bibitem{SchaNSE}
S.~H. Schanuel.
\newblock Negative sets have {E}uler characteristic and dimension.
\newblock In {\em Category Theory (Como, 1990)}, Lecture Notes in Mathematics
  1488, pages 379--385. Springer, Berlin, 1991.

\bibitem{SpeyTPMF}
D.~Speyer.
\newblock Re: {T}utte polynomials and magnitude functions.
\newblock Comments at \cite{NCTPMF}, 2013.

\bibitem{StanEC1}
R.~P. Stanley.
\newblock {\em Enumerative Combinatorics Volume 1}.
\newblock Cambridge Studies in Advanced Mathematics 49. Cambridge University
  Press, Cambridge, 1997.

\bibitem{Tutt}
W.~T. Tutte.
\newblock A contribution to the theory of chromatic polynomials.
\newblock {\em Canadian Journal of Mathematics}, 6:80--91, 1954.

\bibitem{Weil}
A.~Weil.
\newblock A 1940 letter of {A}ndr{\'e} {W}eil on analogy in mathematics.
\newblock {\em Notices of the American Mathematical Society}, 52(3):334--341,
  2005.

\bibitem{Whit2IG}
H.~Whitney.
\newblock 2-isomorphic graphs.
\newblock {\em American Journal of Mathematics}, 55:245--254, 1933.

\bibitem{WillTPMF}
S.~Willerton.
\newblock Re: {T}utte polynomials and magnitude functions.
\newblock Comments at \cite{NCTPMF}, 2013.

\end{thebibliography}
\end{document}